%% file: embed.tex
\documentclass[bibalpha]{amsart}
\usepackage{verbatim}
\usepackage{enumerate}

\def \hdim {\operatorname{dim}}

\def \bR{\mathbb{R}}

\input{makros.tex}
\begin{document}
\title{Isometric embeddability of Snowflakes}

\author{Erik Walsberg}
\address{Institute of Mathematics\\
the Hebrew University of Jerusalem\\
Givat Ram,
Jerusalem, 9190\\
 Israel}
\email{erikw@math.ucla.edu}
\date{\today}
\maketitle
\begin{center}
\textit{To Gabor Elek.}
\end{center}
\begin{abstract}
We show that a snowflake of a metric space with positive Hausdorff dimension does not admit an isometric embedding into euclidean space.
\end{abstract}

\section{Introduction}
Let $(X,d)$ be a metric space.
By euclidean space we mean $\bR^k$ equipped with the standard euclidean metric.
Given $\lambda > 0$ a map $f: X \rightarrow \bR^k$ gives a $\lambda$-bilipschitz embedding of $(X,d)$ into euclidean space if:
$$ \frac{1}{\lambda} d(x,x') \leq \| f(x) - f(x') \| \leq \lambda d(x,x') \quad \text{for all } x,x' \in X. $$
The map $f$ is a bilipschitz embedding if it is a $\lambda$-bilipschitz embedding for some $\lambda > 0$.
Bilipschitz embeddings are injective, so the term ``embedding'' is justified.
The map $f$ is an isometric embedding if it is a $1$-bilipschitz embedding.
Given $0 < r < 1$, the $r$-snowflake of $(X,d)$ is the metric space $(X,d^r)$.
We say that $(X,d^r)$ is a snowflake of $(X,d)$.
Let $K > 0$, we say that $(X,d)$ is $K$-doubling if every open ball of radius $t$ contains at most $K$ pairwise disjoint open balls of radius $\frac{1}{2}t$.
The metric space $(X,d)$ is said to be doubling if it is $K$-doubling for some $K > 0$.
It is easy to see that the following facts hold:
\begin{itemize}
\item Euclidean space is doubling.
\item Any metric space which admits a bilipschitz embedding into euclidean space is doubling.
\item A snowflake of a doubling metric space is doubling.
\end{itemize}
The following marvelous theorem, due to Assouad, gives a kind of converse to the simple facts listed above:
\begin{Thm}[Assouad]
Suppose that $(X,d)$ is doubling and $0 < r < 1$.
Then the $r$-snowflake of $(X,d)$ admits a bilipschitz embedding into some euclidean space.
\end{Thm}
It is natural to wonder when snowflakes admit isometric embeddings into euclidean space.
In this paper we show that this is generally not the case:
\begin{Thm}\label{Thm:it}
Suppose that $(X,d)$ has positive Hausdorff dimension and $0 < r < 1$.
Then the $r$-snowflake of $(X,d)$ does not admit an isometric embedding into euclidean space.
\end{Thm}
We thank Enrico le Donne for useful correspondence on this topic and Samantha Xu for providing an excellent working environment.

\section{Preliminaries}
Our proof depends on some basic geometric facts which we gather in this section.
We begin with an elementary lemma:
\begin{Lem}\label{Lem:general}
Let $x_1,\ldots, x_{l+1} \in \bR^{l}$ be in general position.
Then the map $\sigma: \bR^{l} \rightarrow \bR^{l + 1}$ given by
$$ \sigma(y) = ( \| y - x_1 \|, \ldots, \|y - x_{l+1}\|) $$
is injective.
\end{Lem}
\begin{proof}
We suppose otherwise towards a contradiction.
Suppose that $y,y' \in \bR^{l}$ are such that $y \neq y'$ and $\sigma(y) = \sigma(y')$.
Let $H$ be the set of $x \in \bR^{l}$ such that $\| x - y \| = \| x - y' \|$.
So $x_1,\ldots,x_{l+1} \in H$.
However, as $H$ is a hyperplane of codimension one, this implies that $x_1,\ldots, x_{l+1}$ are not in general position.
\end{proof}
We let $D \subseteq \bR^{l+1}$ be the image of $\sigma$ and let $\tau: D \rightarrow \bR^{l}$ be the compositional inverse of $\sigma$.
That is, if $\bar{t} = (t_1,\ldots,t_{l+1}) \in D$ then $\tau(\bar t)$ is the unique $y \in \bR^{l}$ such that:
$$ \| y - x_i \| = t_i \quad \text{for all } 1 \leq i \leq l+1. $$
We make use of the following:
\begin{Lem}\label{Lem:decomp}
There are smooth submanifolds $D_1,\ldots,D_m \subseteq \bR^{l+1}$ such that $D = D_1 \cup \ldots \cup D_m$ and the restriction of $\tau$ to each $D_i$ is smooth.
\end{Lem}
\begin{proof}
It is presumably easy to prove Lemma~\ref{Lem:decomp} in an elementary way.
However, the present author is a logician.
Therefore, we give a very general proof using semialgebraic geometry.
A set $A \subseteq \bR^k$ is semialgebraic if it is a finite union of sets of the form
$$ \{ \bar x \in \bR^k : p(\bar x) \geq 0\} \quad \text{ for polynomial $p$.} $$
A function $f: A \rightarrow B$ between semialgebraic subsets $A,B \subseteq \bR^k$ is semialgebraic if its graph is a semialgebraic subset of $\bR^k \times \bR^k$.
We refer to \cite{bcr} for information about semialgebraic geometry.
It is well known that every semialgebraic subset of euclidean space is a finite union of smooth submanifolds of euclidean space and if $A \subseteq \bR^k$ and $f: A \rightarrow \bR^n$ are semialgebraic then $A$ can be written as a finite union of smooth submanifolds of $\bR^k$ in such a way that the restriction of $f$ to every submanifold is smooth.
It is an immediate consequence of Tarski-Seidenberg quantifier elimination that $D \subseteq \bR^{l+1}$ and $\tau: D \rightarrow \bR^l$ are both semialgebraic.
Lemma~\ref{Lem:decomp} follows.
\end{proof}
\begin{Lem}\label{Lem:raise}
Let $A \subseteq D$.
The Hausdorff dimension of $\tau(A)$ is no greater then the Hausdorff dimension of $A$.
\end{Lem}
Lemma~\ref{Lem:raise} is a straightforward consequence of Lemma~\ref{Lem:decomp} and a few standard facts about Hausdorff dimension which can be found in \cite{mattila} or other places.
We let $\dim$ be the Hausdorff dimension.
\begin{proof}
Let $D_1,\ldots, D_m$ be as in the statement of Lemma~\ref{Lem:decomp}.
Then:
$$ \dim(A) = \max\{ \hdim (D_i \cap A) : 1 \leq i \leq m \}$$
and 
$$ \hdim( \tau(A)) = \max\{ \hdim(\tau(D_i \cap A)) : 1 \leq i \leq m \}. $$
Smooth maps do not raise Hausdorff dimension, therefore:
$$ \hdim(\tau(D_i \cap A)) \leq \hdim(D_i \cap A) \quad \text{for all } 1 \leq i \leq m. $$
\end{proof}

\section{Proof}
In this section we prove Theorem~\ref{Thm:it}.
We let $(X,d)$ be a metric space with positive Hausdorff dimension and $0 < r< 1$.
Let $D$ and $\tau$ be as in the previous section and let $\hdim$ be the Hausdorff dimension.
We suppose toward a contradiction that $\iota: X \rightarrow \bR^l$ gives an isometric embedding of $(X,d^r)$ into euclidean space.
We may suppose that $\iota(X)$ contains $l+1$ points $y_1,\ldots, y_{l + 1}$ in general position.
If this is not the case then $\iota(X)$ is contained in a hyperplane with positive codimension, and we replace $\iota$ with an isometric embedding into a euclidean space with smaller dimension.
We let $x_1,\ldots, x_{l+1} \in X$ be such that
$$ \iota(x_i) = y_i \quad \text{for all } 1 \leq i \leq l+1. $$
It follows from Lemma~\ref{Lem:general}
that for all $x \in X$, $\iota(x)$ is the unique $y \in \bR^l$ such that
$$ \| y_i - y \| = d(x_i, x)^r \quad \text{for all } 1 \leq i \leq l+1. $$
Let $X' = X \setminus \{x_1,\ldots, x_{l+1}\}$.
Let $U$ be the open subset of $\bR^{l+1}$ consisting of elements with positive coordinates.
Let $f: X' \rightarrow U$ be given by
$$ f(x) = ( d(x_1,x),\ldots, d(x_{l+1},x))$$
and $g: U \rightarrow U$ be given by
$$ g(t_1,\ldots, t_{l+1}) = (t_1^r, \ldots, t_{l+1}^r) $$
Note that $g \circ f$ maps $X'$ into $D$.
The restriction of $\iota$ to $X'$ can be factored as the composition
$$ X' \stackrel{f}{\longrightarrow} U \stackrel{g}{\longrightarrow} U \stackrel{\tau}{\longrightarrow} \bR^l. $$
As $f$ gives a lipschitz map $(X',d) \rightarrow U$ we have $\hdim f(X') \leq \hdim(X',d)$.
As $g$ is smooth it does not raise Hausdorff dimension so $\hdim (g \circ f)(X') \leq \hdim(X',d)$ as well.
It follows from Lemma~\ref{Lem:raise} that $\hdim (\tau \circ g \circ f)(X') \leq \hdim(X',d)$.
Therefore $\hdim \iota(X') \leq \hdim(X',d)$.
As $X \setminus X'$ is finite we have $ \hdim \iota(X) \leq \hdim(X,d)$.
As $\iota(X)$ is isometric to $(X,d^r)$ this implies that $\hdim(X,d^r) \leq \hdim(X,d)$.
However, it follows immediately from the definition of Hausdorff dimension that $\hdim(X,d^r) = \frac{1}{r}\hdim(X,d)$.
This yields a contradiction as $\frac{1}{r} > 1$ and $\hdim(X,d) > 0$.

\bibliographystyle{alpha}
\bibliography{franzi}
\end{document}

%% file: makros.tex
\usepackage{enumerate}
\usepackage{braket}
\usepackage{txfonts}
\newtheorem{Th}{Theorem}[section]
\newtheorem{Thm}[Th]{Theorem}

\newtheorem{Lem}[Th]{Lemma}

\newtheorem*{Lem*}{Lemma}